\documentclass[12pt]{amsproc}
\usepackage{epstopdf}
\usepackage{amsmath}
\usepackage{amsthm}
\usepackage{amsfonts,amssymb}
\usepackage{color}
\usepackage{enumerate}
\usepackage{graphicx}

\newtheorem{thm}{Theorem}[section]

\newtheorem{lem}{Lemma}[section]
\newtheorem{prop}{Proposition}[section]
\theoremstyle{remark}

\numberwithin{equation}{section}

\newcommand\CC{\mathbb{C}}
\newcommand\NN{\mathbb{N}}
\newcommand\RR{\mathbb{R}}
\renewcommand\SS{\mathbb{S}}

\newcommand\proj{\operatorname{proj}}

\newcommand\supp{\operatorname{supp}}
\newcommand\esssup{\operatorname*{ess\,sup}}

\newcommand\wh{\widehat}
\def\hb{\hfill$\Box$}
\newcommand\Cdot{\mathop\cdot}
\def\og{\omega}

\newcommand\HH{\mathcal{H}}

\def\sub{\substack}
\def\sph{\mathbb{S}^{d-1}}
\def\f{\frac}

\def\Bl{\Bigl} \def\Br{\Bigr}

 \def\bl{\bigl}
\def\br{\bigr}
\def\({\left(}
\def \){ \right)}
\def\sa{\sigma}
 \def\a{{\alpha}}
 \def\b{{\beta}}

 \def\t{{\theta}}
 \def\l{{\lambda}}
 \def\ld{\lambda}
 
 \def\o{{\omega}}
 \def\s{{\sigma}}
 
 \def\va{\varepsilon}

 \def\BB{{\mathbb B}}
 \def\CC{{\mathbb C}}

 \def\NN{{\mathbb N}}

 \def\RR{{\mathbb R}}
  
  \def\SS{{\mathbb S}}
 
 \def\al{\a}
 
 \def\proj{\operatorname{proj}}

\def\hb{\hfill$\Box$}

\def\sph{\mathbb{S}^{d-1}}

\def\s{\sa}

\def\ga{\gamma}
\def\Ga{\Gamma}

\def\be{\b}
\def\HH{\mathcal{H}}

\begin{document}

\title[]{ Nikolskii constants for  polynomials \\ on the unit sphere}

\author{Feng Dai}
\address{F.~Dai, Department of Mathematical and Statistical Sciences\\
University of Alberta\\ Edmonton, Alberta T6G 2G1, Canada.}
\email{fdai@ualberta.ca}

\author{Dmitry~Gorbachev}
\address{D.~Gorbachev, Tula State University,
Department of Applied Mathematics and Computer Science, 300012 Tula, Russia}
\email{dvgmail@mail.ru}

\author{Sergey Tikhonov} \address{
S.~Tikhonov,
Centre de Recerca Matem\`{a}tica\\
Campus de Bellaterra, Edifici~C 08193 Bellaterra (Barcelona), Spain; ICREA, Pg.
Llu\'{i}s Companys 23, 08010 Barcelona, Spain, and Universitat Aut\`{o}noma de
Barcelona.} \email{stikhonov@crm.cat}

\thanks{F. ~D. was   supported
 by  NSERC  Canada under the
grant  RGPIN 04702 Dai.  D.~G. was supported by the RFBR (no.~16-01-00308) and the Ministry of
Education and Science of the Russian Federation (no.~5414GZ).
S.~T. was partially supported by  MTM 2014-59174-P and 2014 SGR 289.
}



\begin{abstract}
  This paper
   studies
 the asymptotic
behavior of  the exact   constants of the Nikolskii inequalities  for the space $\Pi_n^d$ of  spherical polynomials  of degree at most $n$ on the unit sphere $\SS^d\subset \RR^{d+1}$ as $n\to\infty$.  It is shown that
 for $0<p<\infty$,
\[
\lim_{n\to \infty}  \sup\Bl\{\f{\|P\|_{L^\infty(\SS^d)}}{n^{\f dp}\|P\|_{L^p(\SS^d)}}:\  \ P\in\Pi_n^d\Br\} =\sup\Bl\{
\f{\|f\|_{L^\infty(\RR^{d})}}{\|f\|_{L^p(\RR^d)}}:\   \ f\in\mathcal{E}_p^d \Br\},   \]
where  $\mathcal{E}_p^d$ denotes  the space  of  all entire functions  of
spherical exponential type at most  $1$ whose restrictions to $\RR^d$  belong to  the space $L^p(\RR^d)$, and it is agreed that $0/0=0$.  It is further proved that  for  $0<p<q<\infty$,
\[
\liminf_{n\to \infty}  \sup\Bl\{\f{\|P\|_{L^q(\SS^d)}}{n^{d(1/p-1/q)}\|P\|_{L^p(\SS^d)}}:\  \ P\in\Pi_n^d\Br\} \ge \sup\Bl\{
\f{\|f\|_{L^q(\RR^{d})}}{\|f\|_{L^p(\RR^d)}}:\   \ f\in\mathcal{E}_p^d\Br\}.   \]
These  results extend the recent    results of Levin and
Lubinsky for trigonometric polynomials on the unit circle.

 The paper also determines   the   exact value of the  Nikolskii constant   for nonnegative functions with  $p=1$ and $q=\infty$:
$$\lim_{n\to \infty}  \sup_{0\leq P\in\Pi_n^d}\f{\|P\|_{L^\infty(\SS^d)}}{\|P\|_{L^1(\SS^d)}} =\sup_{0\leq f\in\mathcal{E}_1^d}\f{\|f\|_{L^\infty(\RR^{d})}}{\|f\|_{L^1(\RR^d)}}
  =\frac1{4^d \pi^{d/2}\Ga(d/2+1)}.$$

\end{abstract}

\date{August 7, 2017}
\maketitle
%
%

\section{Introduction}

Let $\SS^{d}=\{x\in \mathbb{R}^{d+1}\colon |x|=1\}$ denote the unit sphere of
$\RR^{d+1}$ equipped  with the usual surface Lebesgue  measure $d\sigma(x)$, and  $\omega_d$  the  surface area 
of the sphere $\SS^{d}$; that is,  $\omega_{d}:=\s(\SS^d)=2\pi^{\frac{d+1}{2}}/\Gamma(\frac{d+1}{2})$.
 Here, $|\Cdot|$ denotes the Euclidean norm of $\RR^{d+1}$.
 Given $0<p\le \infty$, we denote by $L^p(\SS^{d})$ the
usual Lebesgue $L^p$-space defined with respect to the measure $d\sigma(x)$ on
$\SS^{d}$, and $\|\Cdot\|_p=\|\Cdot\|_{L^{p}(\SS^{d})}$ the quasi-norm of
$L^p(\SS^{d})$; that is,
\[
\|f\|_{p}=\left(\int_{\SS^{d}}|f(x)|^{p}\,d\sigma(x)\right)^{1/p},\quad
0<p<\infty,\qquad \|f\|_{\infty}=\esssup_{x\in \SS^{d}}|f(x)|.
\]
Let $\rho(x,y):=\arccos{}(x\Cdot y)$ denote the geodesic distance between $x,
y\in\SS^d$.
We will use the letter $e$  to denote   the vector  $(0,\dots,0,1)\in\SS^d $.  The notation $ A\asymp B$
means that there exists a positive  constant $c$, called the constant of
equivalence, such that $ c^{-1} A \le B \le c A.$

Let $\Pi_n^d$ denote the space of all spherical polynomials of
degree at most $n$ on $\SS^{d}$ (i.e., restrictions on $\SS^{d}$ of polynomials
in $d+1$ variables of total degree at most $n$), and $\HH_n^d$ the space of all
spherical harmonics of degree $n$ on $\SS^{d}$. As is well known (see, e.g., \cite[Chap.~1]{BOOK}),   both $\HH_n^d$ and $\Pi_n^d$ are  finite
dimensional spaces with
\[
\dim
\HH_n^d=\frac{2n+d-1}{d-1}\,\frac{\Gamma(n+d-1)}
{\Gamma(n+1)\Gamma(d-1)}=
\frac{2n^{d-1}}{\Gamma(d)}\,\bigl(1+ O(n^{-1})\bigr)
\]
and
\begin{equation}\label{dimension}
\dim \Pi_n^d=\frac{(2n+d)\Gamma(n+d)}{\Gamma(n+1)\Gamma(d+1)}=
\frac{2n^{d}}{\Gamma(d+1)}\,\bigl(1+O(n^{-1})\bigr)
\end{equation}
as $n\to\infty$.

 The spaces $\HH_k^d$ are mutually orthogonal with respect to the
inner product of $L^2(\SS^{d})$, and    the orthogonal projection $\proj_k$ of $L^2(\SS^{d})$ onto the space $\HH_k^d$ can be expressed as a spherical convolution:
\begin{equation*}\label{1-1-0}
\proj_kf(x) =\frac{k+\lambda}{\lambda}\f1{\o_d}\int_{\SS^{d}} f(y) C_k^\lambda (x\Cdot
y)\,d\sigma(y),\quad x\in\SS^{d},\   \  \l=\f{d-1}2,
\end{equation*}
where the $C_k^\lambda$ denote  the Gegenbauer polynomials as defined in \cite{Sz}.
As a result, each spherical  polynomial $f\in
\Pi_n^d$ has an integral representation,
\begin{equation*}
f(x)=\int_{\SS^{d}}G_{n}(x\Cdot y)f(y)\,d\sigma(y),
\end{equation*}
where
\begin{equation}\label{reproduce}
G_{n}(t)=\f1{\o_d}\sum_{k=0}^n \f{k+\ld}\ld C_k^{\ld}(t)=d_{n}R_n^{(\frac{d}{2},\frac{d-2}{2})}(t),
\end{equation}
$R_{n}^{(\alpha,\beta)}(t)=\frac{P_{n}^{(\alpha,\beta)}(t)}{P_{n}^{(\alpha,\beta)}(1)}$
 denotes the normalized Jacobi polynomial, and  $d_{n}:=\dim \Pi_n^d/\o_d$.

The classical \textit{Nikolskii inequality} for spherical polynomials reads as
follows (see, e.g., \cite{Kam}):
\begin{equation}\label{Nikol}
\|f\|_q \le C_{d}n^{d\,(\frac 1p-\frac 1q)}\|f\|_p\qquad \forall\,f\in\Pi_n^d,\quad
0<p<q\le \infty.
\end{equation}
 In the case when $0<p\leq 2$ and $0<p<q\leq \infty$, the constant (not the optimal one) in \eqref{Nikol} can be written explicitly  (see, for instance, \cite{BDD, ul}):
\begin{equation}\label{C-p-2}
\|f\|_q \le (d_{n})^{1/p-1/q}
\|f\|_p\qquad \forall\,f\in\Pi_n^d,
\end{equation}
where $d_n=\o_d^{-1}\text{dim}\ \Pi_n^d$.

Our main interest in this paper is the following Nikolskii constant:
\begin{equation}\label{1-5}
    C(n,d,p,q):=\sup\Bigl\{\|f\|_{L^q(\SS^{d})}
    \colon\ f\in\Pi_n^{d}\quad
\text{and}\quad \|f\|_{L^p(\SS^{d})}=1\Bigr\},\   \ 0<p<q\leq \infty.
\end{equation}
By log-convexity of the $L^p$-norm, it is easily seen   that if $0<p<q<q_1\leq \infty$,  then
$$ C(n, d, p, q) \leq C(n, d, p, q_1)^{\f { 1/p- 1/q}{1/p-1/{q_1}}}.$$
 Also, note that  according to  \eqref{C-p-2},  if $0<p\leq 2$ and $p<q$, then (see also \cite{Dei09})
\begin{equation*}\label{C1-upper-est}
C(n,d, p, q)\le d_{n}=\Bl(\frac{1}{2^{d} \Ga(d/2+1)\pi^{d/2}}\Br)^{ 1/p-1/q}\,n^{d( 1/p-1/q)}\bigl(1+O(n^{-1})\bigr).
\end{equation*}

 The asymptotic order in the Nikolskii inequality \eqref{Nikol} or \eqref{C-p-2}   is sharp  in the sense that $C(n,d,p,q)\asymp  n^{d(\f 1p-\f1q)}$ for $0<p<q\leq \infty$ as $n\to \infty$
with the constant of equivalence depending only on $d$ and $p$ when $p\to 0$.
However,
the exact value of the sharp constant $C(n,d,p,q)$ is known only   in the case when $p=2$ and $q=\infty$, where a simple  application of  the addition formula for spherical harmonics leads to
\begin{equation}\label{1-8L2}
    C(n, d, 2, \infty)=\sqrt{d_n}.
\end{equation}
For $(p, q)\neq (2, \infty)$, the constant in  \eqref{C-p-2} is not optimal.
It is a longstanding  open  problem to determine the exact value of the  constant $C(n,d,p,q)$ for $(p, q)\neq (2,\infty)$ and $0<p<q\leq \infty$.  This problem is  open even for  trigonometric polynomials on the unit circle (i.e., the case of $d=1$). We refer to \cite{AD14, Gor05} for  historical background on this problem.

Of related interest is a recent  result of Arestov and Deikalova \cite{AD14} showing that the supremum in \eqref{1-5} can be in fact  achieved by zonal  polynomials for $q=\infty$. More precisely, they proved that
\begin{equation}\label{C-P}
C(n,d, p, \infty)=\sup_{
 \deg P\leq n}\frac{P(1)}{\Bl( {\o_{d-1}}\int_{-1}^1 |P(t)|^p(1-t^2)^{(d-2)/2}\, dt\Br)^{1/p}},\   \ 0<p<\infty
\end{equation}
with the supremum being taken over all real  algebraic polynomials $P$ of degree at most $n$ on $[-1,1]$.

In this paper, we will  study the asymptotic behavior of the quantity  $\f{C(n,d,p,q)}{n^{d(1/p-1/q)}}$ as $n\to\infty$.  Our work was  motivated by  a recent   work of Levin and
Lubinsky \cite{LL1, LL2}, who proved (using the notation of the current paper)\footnote{ Trigonometric polynomials in \cite{LL1, LL2} are written in the form $P(e^{it})$ with $P$ being an algebraic polynomial of degree $n$ on $[-1,1]$. Note that  the absolute value $|P(e^{it})|$ corresponds to the absolute value of a trigonometric polynomial of degree at most $(n+1)/2$. } that for $d=1$,
\begin{equation*}\label{}
  \lim_{n\to\infty}  \f { C(n, 1, p, \infty)}{ n^{ 1/p}}=\mathcal{L}( p, \infty), \  \  0<p<\infty,
\end{equation*}
and
\begin{equation*}\label{}
  \liminf_{n\to\infty}  \f { C(n, 1, p, q)}{ n^{ 1/p-1/q}}\ge \mathcal{L}(p,q),\   \  0<p<q<\infty.
\end{equation*}
Here  the constant $\mathcal{L}(p,q)$ is defined as
$$\mathcal{L}(p,q):=\sup\f{\|f\|_{L^q(\RR)}}{\|f\|_{L^p(\RR)}},\   \   0<p<q\leq \infty$$
with the supremum being taken over all entire functions of exponential type at most $1$.
For more related results in one variable, we also  refer to  \cite{Gor05, gt}.

Our main goal in this paper is to   extend these  results of  Levin and
Lubinsky to the high dimensional cases. To be more precise,  recall that
an entire function $F$ of $d$-complex variables is  of spherical exponential type at most $\sa>0$ if for every $\va>0$ there exists a constant $A_\va>0$ such that
$|F(z)|\leq A_\va e^{(\sa+\va)|\text{Im} (z)|}$ for all $z=(z_1, \cdots, z_d)\in\CC^d$. Given $0<p\leq \infty$, we denote by  $\mathcal{E}_{p}^{d}$ the class of all entire functions
of  spherical exponential type at most  $1$ on $\CC^d$ whose restrictions to $\RR^d$  belong to  the space $L^p(\RR^d)$, (see, for instance, \cite[Ch.~3]{Nik75} and \cite{grafakos}).

According to the Palay-Wiener theorem (\cite[Subsect.~3.2.6]{Nik75}),   each  function  $f\in \mathcal{E}_{p}^{d}$ can be identified with a function  in $ L^p(\RR^d)$ whose
 distributional Fourier transform is  supported in the unit ball $\mathbb{B}^{d}=\{x\in \RR^{d}\colon
|x|\le 1\}$.  Here  we also  recall that the Fourier transform of $f\in L^1(\RR^d)$ is defined by
\[
\mathcal{F} f(\xi)\equiv \wh{f}(\xi)=\int_{\RR^d} f(x) e^{- i x\Cdot \xi}\, dx,\quad \xi\in\RR^d,
\]
while   the inverse Fourier transform is given by
\[
\mathcal{F}^{-1} f(x)=\frac 1{(2\pi)^d}\int_{\RR^d} f(\xi) e^{ix\Cdot \xi}\, d\xi,\quad f
\in L^1(\RR^d),\  \  x\in\RR^d.
\]
As is well known, if  $0<p<q\leq \infty$, then  $\mathcal{E}_p^d\subset \mathcal{E}_q^d$
 and there exists a constant $C=C_{d,p,q}$ such that $\|f\|_q \leq C \|f\|_p$ for all $f\in\mathcal{E}_p^d$.
For $0<p<q\leq \infty$,
let $\mathcal{L}(d,p,q)$  denote the sharp Nikolskii constant  defined by
\[
\mathcal{L}(d,p,q):=\sup\{\|f\|_{L^q(\RR^d)}\colon\ f\in\mathcal{E}_{p}^{d}\ \
\text{and}\ \ \|f\|_{L^p(\RR^d)}=1\}.
\]
  Recall also  that the constant $C(n, d, p, q)$ is defined in \eqref{1-5}.

In this paper, we will prove the following theorem, which extend a recent  result of  Levin and
Lubinsky \cite{LL1, LL2}:
\begin{thm}\label{thm-3-1}\begin{enumerate}[\rm (i)]
\item For $0< p<\infty$, we have
\begin{equation*}\label{}
\lim_{n\to \infty} \frac{C(n,d, p,\infty)} {n^{d/p}}=\mathcal{L}(d, p, \infty).
\end{equation*}

\item For $0< p<q\le \infty$,
\[
  \liminf_{n\to \infty} \frac{C(n,d,p,q)}{n^{d(1/p-1/q)}}\ge
  \mathcal{L}(d,p,q).
\]
\end{enumerate}
\end{thm}

Note that as an immediate consequence of \eqref{1-8L2} and Theorem \ref{thm-3-1}, we obtain
\[
\mathcal{L}(d,2,\infty)
=
\Bigl(\frac{2}{\o_d\Gamma(d+1)}\Bigr)^{1/2}.
\]

Compared with those  in \cite{LL1, LL2} and \cite{Gor05, gt} in one variable, the proof of Theorem~\ref{thm-3-1} in more variables  is   fairly nontrivial because :  1) functions on the sphere can not be identified as periodic functions on  Euclidean space; 2) explicit  connections between  spherical  polynomial interpolation $\SS^d$ and the Shannon sampling theorem for entire functions of exponential type are not available.
 Our proof relies on a  recent deep result of
Bondarenko, Radchenko and Viazovska \cite{BRV1, BRV2} on spherical designs.

 It is a longstanding open problem to determine  the exact value of the Nikolskii constant $\mathcal{L}(d, p, \infty)$  even for $p=1$ and $d=1$.
  In this paper, we find  the exact value of the Nikolskii constant $\mathcal{L}(d, 1, \infty)$
for nonnegative functions. Our main result in this direction  can be stated as follows:
\begin{thm}\label{thm1-2} We have
$$
\lim_{n\to \infty}  \sup_{\sub{0\leq P\in\Pi_n^d\\
\|P\|_{L^1(\SS^d)}=1}}n^{-d}\|P\|_{L^\infty(\SS^d)} =\sup_{\sub{0\leq f\in\mathcal{E}_1^d\\
\|f\|_{L^1(\RR^d)}=1}}\|f\|_{L^\infty(\RR^{d})}
  =
\frac1{(4\sqrt{\pi})^d \Ga(d/2+1)}.$$
\end{thm}

  It is worthwhile to point out that the exact Nikolskii constant for  nonnegative polynomials with $p=1$ and  $q=\infty$   has interesting applications in  metric geometry. For instance, it was used  to
obtain some  tight-bounds for spherical designs in  \cite{BRV1,Lev98}.

This paper is organized as follows. Section 2 contains several  preliminary lemmas,  which will play an important role in the  proof of Theorem \ref{thm-3-1}.
Section 3 is  devoted to the proof of the   lower estimate of  Theorem \ref{thm-3-1},
\begin{equation*}
  \liminf_{n\to \infty} \frac{C(n,d,p,q)}{n^{d(1/p-1/q)}}\ge
  \mathcal{L}(d,p,q),\   \   0< p<q\le \infty,
\end{equation*}
 whereas the corresponding  upper estimate,
\begin{equation*}
  \limsup_{n\to\infty}\frac{C(n,d, p, \infty)}{n^{d/p}}\le \mathcal{L}(d, p,\infty),\   \   \  0<
p<\infty,
\end{equation*}
is proved in Section 4.
Finally, in Section 5, we prove  Theorem \ref{thm1-2}.


Throughout the paper,  all functions  are  assumed to be real-valued and Lebesgue measurable unless otherwise stated, and  we denote by  $B(r)$ the ball in $\RR^d$ centered at origin having   radius $r>0$.

\section{Preliminary lemmas}

 In this section, we will present   a few preliminary lemmas that will be used in the proof of  Theorem \ref{thm-3-1}.

  We start with the following well-known property of the  Geigenbauer polynomials.

\begin{lem}\cite[(8.1.1), p.192]{Sz}\label{lem-2-3-0} For $z\in\mathbb{C}$ and $\mu\ge 0$,
\begin{equation*}\label{2-5}
\lim_{k\to\infty} \f{ C_k^\mu \(\cos \f
{z}k\)}{C_k^\mu(1)}=
j_{\mu-1/2} (z),\end{equation*} where $j_\al (z) =\Ga(\mu+\f12)(z/2)^{-\al} J_\al (z)$, and $J_\al$ denotes the Bessel function of the first kind.  This formula holds uniformly in
every bounded region of the complex $z$-plane.\end{lem}

Next, we note that  a function on the sphere $\SS^d$ in general cannot be identified   with a periodic function on $\RR^d$, which is different from the one-dimensional case.   In our next lemma,   we   connect functions on $\SS^d$ with functions on $\RR^d$ via the following  mapping $\psi : \RR^d
\to
\SS^d$:
\[
\mathbf{\psi}(x) := \left(\xi \sin |x|, \cos |x|\right)\quad
\text{for \ $x=|x|\xi\in \RR^d$ \ and \ $\xi\in \SS^{d-1}$}.
\]
It is easily seen that $\psi\colon B(\pi)\to \SS^d$ is a bijective mapping and $\rho(\psi(x),e)=|x|$ for all $x\in B(\pi)$.
Furthermore, for each $f\in L^1(\SS^{d})$,
\begin{align*}
  \int_{\SS^d} f(x)\, d\sigma(x)&=\int_0^\pi
  \Bigl[\,\int_{\SS^{d-1}} f(\xi \sin \theta, \cos \theta)\, d\sa_{d-1}(\xi)\Bigr]
  \Bigl(\frac{\sin\theta}{\theta}\Bigr)^{d-1}\theta^{d-1} \, d\theta\\
  & =\int_{B(\pi)} f(\psi(x)) \Bigl(\frac{\sin |x|}{|x|}\Bigr)^{d-1}\,
  dx,
\end{align*}
where $d\sa_{d-1}$ denotes the usual  surface Lebesgue measure on $\SS^{d-1}$.
As a result,
 we may identify each function $f$ on the ball $B(n\pi)\subset \RR^d$ with a
function $f_n$ on the sphere $\SS^d$ via dilation and the mapping $y=\psi(x/n)$
for each $x\in B(n\pi)$. Indeed, we have
\begin{lem}\label{lem-2-1}\cite{DW} For $n\in\NN$ and  $f\in L^1(\SS^d)$,
\begin{equation}\label{2-2-t}
\int_{\SS^d} f ( x)\, d\sigma(x) = \frac 1{n^d} \int_{B(n\pi)}
f(\psi(x/n)) \left( \frac{\sin (|x|/n)}{|x|/n}\right)^{d-1}\, dx.
\end{equation}
\end{lem}

Note that in the case of $d=1$, \eqref{2-2-t} becomes
\begin{equation*}
\int_{\SS^1} f ( x)\, d\sigma(x) = \frac 1{n} \int_{- n\pi}^{n\pi}
f(\sin \f\t n, \cos \f \t n) d\t.
\end{equation*}

Our  last  preliminary  lemma can be stated as follows.
\begin{lem}\label{lem-3-5} Let $\eta$ be a $C^\infty$-function on $[0,\infty)$
that is supported on $[0,2]$ and is constant near $0$.
For a positive integer $n$, define
\[
 G_{n,\eta}( t ):=\f 1{\o_d}
\sum_{j=0}^{2n}\eta(n^{-1}j)\,\frac{j+\ld}{\mu}
\,C_j^{\ld}(t), \quad
t\in [-1,1],
\]
where $\ld=\frac{d-1}2$. Then for any $u, v\in\SS^d$, $n\in\NN$ and any $\ell>0$,
\begin{equation}\label{3-5-0}
 |G_{n,\eta}(u\cdot v)|\le C_{d, \eta, \ell} n^d (1+n\rho(u,v))^{-\ell}.
\end{equation}
Furthermore,
\begin{equation}\label{3-5}
  \lim_{n\to\infty} \f 1{n^d} G_{n,\eta}\Bigl( \psi\Bigl(\frac
  xn\Bigr)\Cdot \psi\Bigl(\frac yn\Bigr)\Bigr)= K_{\eta}(|x-y|)
\end{equation}
holds uniformly on every compact subset of $(x, y)\in \RR^d\times \RR^d$,
where $K_{\eta}(|\Cdot|)$
denotes the inverse Fourier transform of {the radial function}
$\eta(|\Cdot|)$ on $\RR^d$.
\end{lem}

Using the formula for the Fourier transforms of radial functions, we have
\begin{equation}\label{K-eta}
  K_{\eta}(|x|)=\frac{\omega_{d-1}}{(2\pi)^{d}} \int_0^2 \eta(\rho)
  j_{d/2-1}\bigl(\rho |x|\bigr)\rho^{d-1}\, d\rho,\quad x\in\RR^d.
\end{equation}

\begin{proof}
\eqref{3-5-0} is known (see \cite{BD}). We only need to prove
\eqref{3-5}. The proof  is very close to that in \cite{DW}. But for completeness, we include a detailed proof here.
Write
\begin{equation}\label{3-6-0}
n^{-d} G_{n,\eta}\Bigl( \psi\Bigl(\frac xn\Bigr)\Cdot\psi\Bigl(\frac
yn\Bigr)\Bigr)=\int_0^2 b_n (\rho, x, y)\rho^{d-1}\, d\rho,
\end{equation}
where
\[
b_n (\rho, x, y) =n^{-d}\f1{\o_d}\sum_{j=0}^{2n-1}
\eta(n^{-1}j)\,\frac{j+\ld}{\ld}\,C_j^{\ld}\left(\psi\Bigl(\frac xn\Bigr)\Cdot
\psi\Bigl(\frac yn\Bigr)\right) \biggl( \int_{ \frac jn}^{\frac{j+1}n}
t^{d-1}\, dt \biggr)^{-1}\chi_{[\frac jn, \frac{j+1}n)}(\rho),
\]
where $\chi_{I}$ is the characteristic function of the set $I$.
 We first claim that
\begin{equation}\label{3-6-1}
\sup_{x,y\in\RR^d} |b_n (\rho, x, y)| \le c_d\qquad
\forall\,\rho\in [0,2],\quad n=1,2,\dots.
\end{equation}

Indeed, if $0\le \rho< n^{-1}$, then \eqref{3-6-1} holds trivially. Now assume
that $0<\rho\le 2$ and $n>\rho^{-1}$. Let $1\le j\le 2n-1$ be an integer such
that $ \frac jn\le \rho < \frac{j+1}n$. Then
\[
\int_{\frac jn } ^{\frac{j+1} n} t^{d-1}\, dt \ge c n^{-1} \rho ^{d-1},
\]
and hence
\begin{align*}
|b_n (\rho, x, y)|&\leq c n^{-d}\,\frac{j+\ld}{\ld}\Bigl| C_j^\ld \left(
\psi\Bigl(\frac xn\Bigr)\Cdot \psi\Bigl(\frac yn\Bigr)\right)\Bigr|
\biggl( \int_{\frac jn}^{\frac{j+1}n} t^{d-1}\, dt \biggr)^{-1}\\
&\le c n^{-(d-1)} \rho^{-(d-1)} j^{d-1} \le c (n\rho)^{-(d-1)} j^{d-1}\le c.
\end{align*}
This shows the claim \eqref{3-6-1}.

Next, we show that, for any $\rho\in (0,2]$ and any $M>1$,
\begin{equation}\label{2-8}
  \lim_{n\to\infty} \sup_{|x|,|y|\le M}\,\Bigl|b_n (\rho, x,
y)-\frac{2}{\o_d\Gamma(d)}\,\eta(\rho) j_{d/2-1}\bigl(\rho
|x-y|\bigr)\Bigr|=0.
\end{equation}
Combining \eqref{2-8} with  \eqref{3-6-1}, \eqref{3-6-0} and \eqref{K-eta}, and observing that
\begin{equation*}\label{omega-d-d-1}
\o_d\o_{d-1} =\f {2 (2\pi)^d}{\Ga(d)},
\end{equation*}
 we will deduce  the desired equation \eqref{3-5} by dominated convergence theorem.

To show \eqref{2-8}, we assume that $|x|, |y|\le M$. All the constants in the proof below are
independent of $x, y$, but may depend on $M$. Let $n>\rho^{-1}$ and assume that
$\frac jn< \rho \le \frac{j+1} n $ with $1\le j\le 2n-1$. A straightforward
calculation then shows that
\[
\biggl( \int_{\frac jn}^{\frac{j+1}n} t^{d-1} \, dt \biggr)^{-1} = \frac n
{\rho^{d-1} } \left( 1+ O\bigl((n\rho)^{-1}\bigl)\right)\quad \text{as
$n\to\infty$}.
\]
This implies that for $ \frac jn\le \rho \le \frac{j+1} n $ with $1\le j\le
2n-1$,
\begin{align*}
b_n (\rho, x, y)
&= \frac{j+\ld}{(n\rho)^{d-1}\ld\o_d}\,\eta(\rho) C_j^{\frac{d-1}2}
\left( \psi \Bigl(\frac xn\Bigr)\Cdot \psi \Bigl(\frac yn\Bigr)\right) (
1+O((n\rho)^{-1}))\\
&= \frac 2{\o_d\Gamma(d)}\,\eta(\rho)\,\frac{C_j^{\frac{d-1}2} (\cos\theta_n(x,y)
)}{C_j^{\frac{d-1}2}(1)}+O(1)j^{-1},
\end{align*}
where we used the formula $C_j^\ld(1)=\f {\Ga(j+2\ld)}{\Ga(j+1)\Ga(2\ld)}$ in the last step, and $\theta_n (x, y)\in [0,\pi]$ satisfies
\[
\cos\theta_n (x, y)=\psi \Bigl(\frac xn\Bigr)\Cdot \psi \Bigl(\frac
yn\Bigr)=\frac{x\Cdot y} {|x|\,|y|} \sin \frac{|x|}n \sin \frac{|y|}n+\cos
\frac{|x|}n\cos \frac{|y|}n.
\]
It is easily seen that
\[
\cos\theta_n ( x, y) =1-\frac{|x-y|^2}{2n^2} +O(n^{-4}).
\]
Hence,
\begin{align*}
\theta_n(x,y)
&=\frac 1n \sqrt{|x-y|^2+O(n^{-2})}
+O (n^{-2})=\frac{|x-y|}n+O(n^{-2})\\
&=\frac{\rho |x-y|+O(j^{-1})}j.
\end{align*}

Recalling that $j\asymp n\rho \to \infty$ as $n\to\infty$, using Lemma \ref{lem-2-3-0},
we obtain that
\[
\lim_{n\to\infty} b_n(\rho, x, y) = \frac 2{\o_d\Gamma(d)}\lim_{n\to \infty}\eta(\rho)
\frac{C_j^{\frac{d-1}2} \bigl(\cos \frac{\rho |x-y|+O(j^{-1})}j
\bigr)}{C_j^{\frac{d-1}2}(1)}=\frac{2}{\o_d\Gamma(d)}\,\eta(\rho) j_{d/2-1} \left(\rho
|x-y|\right),
\]
which shows \eqref{2-8}.
\end{proof}

\section{Proof of Theorem \ref{thm-3-1}: Lower estimate}

This section is  devoted to the proof of the following  lower estimate:
\begin{equation}\label{lower-estimate}
  \liminf_{n\to \infty} \frac{C(n,d,p,q)}{n^{d(1/p-1/q)}}\ge
  \mathcal{L}(d,p,q),\   \   0< p<q\le \infty.
\end{equation}
  The proof requires the use of certain ``maximal functions'' for entire functions of exponential type given in the following   lemma:

\begin{lem}\label{lem-2-3}\cite{HW}
If $f\in\mathcal{E}_p^d$,  $0<p<\infty$ and  $\ell>d/p$, then
$$\|f_\ell^\ast \|_p\leq C_{p} \|f\|_p,$$
where $f_\ell^\ast(x):=\sup_{y\in\RR^d} \f{|f(y)|}{(1+|x-y|)^{\ell}} $ for $x\in\RR^d$.
\end{lem}

Lemma \ref{lem-2-3} for $d=1$ is a direct consequence of Lemma 3.5 and  Corollary 3.9 of \cite[p. 269-271]{HW}, where  the proof   with slight modifications works equally well for the case $d\ge 2$.

Now we turn to the proof of \eqref{lower-estimate}.  Setting
\[
L_{pq}:=\liminf_{n\to\infty}\frac{C(n,d,p,q)}
{n^{d(1/p-1/q)}},
\]
we reduce to showing that
\begin{align}\label{3-6}
\|f\|_q\le L_{pq}\|f\|_p, \qquad \forall\,f\in\mathcal{E}_{p}^{d}.
\end{align}

We first assert that it is enough to prove \eqref{3-6} under  the  additional assumption  that $\supp
\wh{f}\subset B(1-\varepsilon)$ for some $\varepsilon\in (0,1)$.
Indeed, if  $f\in\mathcal{E}_{p}^{d}$, then for  any $\va>0$,
\[
\supp \Bigl(\wh{f}\,\Bigl(\frac{\Cdot}{1-\varepsilon}\Bigr)\Bigr)=\supp
\wh{f_\varepsilon}\subset B(1-\varepsilon),
\]
where $f_\varepsilon(x)=(1-\varepsilon)^d f((1-\varepsilon) x)$. Thus,
applying \eqref{3-6} to $f_\va$  instead of $f$ yields
\[
(1-\varepsilon)^{d/{q'}} \|f\|_q=\|f_\varepsilon\|_q\le
L_{pq}\|f_\varepsilon\|_p=L_{pq}(1-\varepsilon)^{d/{p'}}\|f\|_p,
\]
where $\frac 1q+\frac 1{q'}=1$. \eqref{3-6} for general $f\in\mathcal{E}_p^d$ then follows by
 letting $\varepsilon\to 0$. This proves the assertion.

 For the rest of the proof, we assume that $f\in\mathcal{E}_p^d$ and satisfies   $\supp
\wh{f}\subset B(1-\varepsilon)$ for some $\varepsilon\in (0,1)$.
We will  prove \eqref{3-6} under this extra condition.
  Let $\eta\in
C^\infty[0,\infty)$ be such that $\eta(t)=1$ for $t\in [0, 1-\varepsilon]$ and
$\eta(t)=0$ for $t\ge 1$. As in Lemma \ref{lem-3-5}, we  denote by
$K_\eta(|\Cdot|)$  the inverse Fourier transform of the function
$\eta(|\Cdot|)$.
 Then $K_\eta(|\Cdot|)$ is a Schwartz function on $\RR^d$, and  since
$\wh{f}(\xi)=\wh{f}(\xi)\eta(|\xi|)$, we have
\begin{equation}\label{3-7}
  f(x) =f* K_\eta(x)=\int_{\RR^d} f(y) K_\eta(|x-y|)\, dy,\quad x\in\RR^d.
\end{equation}

Let $m>1$ be a temporarily fixed parameter. For $n>m$, define   $f_{n,m}$ to be a   function on $\SS^{d}$ supported  on the spherical cap $\{ x\in \SS^d\colon
\rho(x, e)\le \frac{m}{n} \}$ and  such that
\[
f_{n,m}\bigl(\psi (x/n)\bigr) = f(x)\chi_{B(m)}(x)\quad \forall\,x\in B(n\pi). \]
Consider the  spherical polynomial $P_{n,m}\in \Pi_n^d$ given by
\begin{equation}\label{3-9-}
  P_{n,m}(x):=\int_{\SS^d} f_{n,m} (y) G_{n,\eta}(x\Cdot y)\, d\sigma(y),\quad
  x\in \SS^d
\end{equation}
with
\[
G_{n,\eta} (t)=\f 1{\o_d}\sum_{j=0}^n \eta\Bigl(\frac jn\Bigr) \frac{j+(d-1)/2}{(d-1)/2}\,
C_j^{\frac{d-1}2}(t).
\]
 By Nikolskii's  inequality \eqref{1-5},
\begin{equation}\label{3-9}
  \|P_{n,m}\|_{L^q (\SS^d)}\le C (n,d,p,q) \|P_{n,m}\|_{L^p (\SS^d)} .
\end{equation}
Moreover, using \eqref{2-2-t}, we have that for any $x\in B(n\pi)$
\begin{align}
  P_{n,m}(\psi(x/n))&=\int_{\SS^d} f_{n,m} (u) G_{n,\eta}(\psi(x/n)\Cdot u)\,
  d\sigma(u)\notag\\
  &=\frac 1{n^{d}}\int_{B(m)} f(y) G_{n,\eta}\Bigl(\psi(x/n)\Cdot
  \psi(y/n)\Bigr)\Bigl(\frac{\sin (|y|/n)}{ |y|/n}\Bigr)^{d-1}\,
  dy.\label{3-16-0}
\end{align}

 We now break the proof of \eqref{3-6} into several parts:

\smallbreak
\underline{Step 1.}
 We show that for any $m\in\NN$,
\begin{equation}\label{3-10}
  \lim_{n\to\infty} \sup_{x\in B(2m)}\Bigl|P_{n,m} \Bigl(\psi\Bigl(\frac
  xn\Bigr)\Bigr)-\int_{B(m)} f(y) K_\eta(|x-y|)\, dy \Bigr|=0.
\end{equation}
This combined  with \eqref{3-7}, in particular, implies that
\begin{equation}\label{3-11}
\lim_{m\to\infty} \limsup_{n\to \infty} |P_{n,m} (\psi(x/n))-f(x)|=0\quad
\forall\,x\in\RR^d.
\end{equation}

To show \eqref{3-10}, we use \eqref{3-16-0} to obtain
\begin{align*}
P_{n,m} (\psi(x/n))
&=\int_{B(m )}
K_\eta(|x-y|)
f(y) \, dy+R_{n,1}(x)+R_{n,2}(x),
\end{align*}
where
\begin{align*}
|R_{n,1}(x)|&\le \int_{B(m)} |f(y)| \Bigl|\f 1{n^d }G_{n,\eta}(\psi\Bigl(\frac
xn\Bigr)\Cdot \psi\Bigl(\frac yn\Bigr))-K_\eta(|x-y|)\Bigr|\,dy,\\
|R_{n,2}(x)|&\le C\int_{B(m)}|f(y)||K_\eta(|x-y|)|\Bigl[1-\left( \frac{\sin
(|y|/n)}{|y|/n}\right)^{d-1}\Bigr]\,dy.
\end{align*}
By either Nikolskii's inequality, $\|f\|_1\leq C_p\|f\|_p$ for  $0<p<1$, or H\"older's inequality if $p\ge 1$,
\[
|R_{n,2}(x)|\le C_m \|f\|_p\sup_{|y|\le m}\Bigl[1-\left( \frac{\sin
(|y|/n)}{|y|/n}\right)^{d-1}\Bigr],
\]
which goes to zero uniformly as $n\to\infty$. On the other hand, it follows by
Lemma \ref{lem-3-5} and the dominated convergence theorem that
\begin{align*}
&\limsup_{n\to\infty}\sup_{x\in B(2m)} |R_{n,1}(x)|\\
&\qquad \le \Bl(\int_{B(m)} |f(y)|\,dy\Br) \lim_{n\to\infty} \Bigl(\,\sup_{u, v\in
B(2m)}\Bigl|\frac 1{n^{d}}\,G_{n,\eta} \Bigl(\psi\Bigl(\frac
un\Bigr)\Cdot \psi\Bigl(\frac v n\Bigr)\Bigr)-K_{\eta}(|u-v|)\Bigl|\Bigr) =0.
\end{align*}
This proves \eqref{3-10}.

\smallbreak
\underline{Step 2.} Prove that for any $\ell>1$,
\begin{equation}\label{step2}
  n^d\int_{\rho(x, e) \ge \frac{2m}n}|P_{n,m}(x)|^p\,
  d\sigma(x)\le C m^{-\ell p}\|f\|_p^p.
\end{equation}

For $x\in\SS^d$ such that $\rho(x, e)\ge
\frac{2m}n$,  write  $x=\psi(u/n)$ with $2m\leq |u|\leq n\pi$. Since $f_{n,m}$ is supported in the spherical cap $\{ y\in \SS^d\colon
\rho(y, e)\le \frac{m}{n} \}$, using \eqref{3-9-} and
Lemma~\ref{lem-2-1} with $\ell >d(1+1/p)$, we obtain that
\begin{align*}
  |P_{n,m}(x)|&\le  \int_{\rho(y, e)\le \frac mn}
  |f_{n,m}(y)| |G_{n,\eta}(x\Cdot y)|\, d\sigma(y)\\
  &\le C n^{d} (1+n\rho(x, e))^{-2\ell-d-1} \int_{\SS^d}
  |f_{n,m}(y)|\, d\sigma(y)\\
  &\leq C m^{-\ell} \int_{|v|\leq m} |f(v)| (1+|u-v|)^{-\ell-d-1}\, dv\leq C m^{-\ell} f_\ell^\ast (u).
\end{align*}
Integrating over the domain $\{x\in\SS^d\colon \rho(x, e)\ge \frac{2m}{n}\}$
then yields
\begin{align*}
  n^d\int_{\rho(x, e)\ge \frac{2m}n}|P_{n,m}(x)|^p\, d\sigma(x)&\le
  C \int_{2m\leq |u|\leq n\pi}|P_{n,m}(\psi(u/n))|^p\, du\\
  &\leq C m^{-\ell p} \int_{\RR^d}|f_\ell^\ast (u)|^p\, du \leq C m^{-\ell p}\|f\|_p^p,
\end{align*}
where the last step uses Lemma \ref{lem-2-3}.
\smallbreak
\underline{Step 3.} Show that for each fixed $m\ge 1$ and any $\ell>1$,
\begin{align}
  &\limsup_{n\to \infty}\Bigl(n^d \int_{\rho(y, e) \le
  \frac{2m}n}|P_{n,m}(y)|^p\, d\sigma(y)\Bigr)^{p_1/p}\notag\\
  &\qquad \le (1+C m^{-\ell})\|f\|_p^{p_1}
  +C \Bigl(\int_{|x|\ge m/2} |f_\ell^\ast(x)|^p\, dx\Bigr)^{p_1/p}
  \label{3-14}
\end{align}
where $p_1=\min\{p,1\}$.

Indeed, using \eqref{2-2-t}, we have
\begin{align*}
&\Bigl(n^d\int_{\rho(y, e) \le \frac{2m}n}|P_{n,m}(y)|^p\,
d\sigma(y)\Bigr)^{p_1/p}=\Bigl(\int_{B(2m)} |P_{n,m}
(\psi(x/n))|^p\Bigl(\frac{\sin (|x|/n)}{|x|/n}\Bigr)^{d-1}\, dx\Bigr)^{p_1/p}\\
&\qquad \le \Bigl(\int_{B(2m)} \Bigl|P_{n,m} (\psi(x/n))-\int_{B(m)} f(y)
K_\eta(|x-y|)\, dy\Bigr|^p\,dx\Bigr)^{p_1/p}\\
&\qquad +\Bigl(\int_{B(2m)} \Bigl| \int_{B(m)} f(y)K_\eta(|x-y|)\,
dy\Bigr|^p\,dx\Bigr)^{p_1/p}=: I_{n,m}+J_{n,m}.
\end{align*}
For the first term $I_{n,m}$, we have
\[
I_{n,m}\le C m^{p_1d/p} \sup_{x\in B(2m)} \Bigl|P_{n,m} (\psi(x/n))-\int_{B(m)}
f(y) K_\eta(|x-y|)\, dy\Bigr|^{p_1},
\]
which, using \eqref{3-10}, goes to zero as $n\to \infty$. For the second term
$J_{n,m}$, we use  \eqref{3-7} to obtain
\begin{align*}
  J_{n,m}&=\Bigl(\int_{B(2m)} \Bigl|f(x)- \int_{|y|\ge m} f(y)K_\eta(|x-y|)\,
  dy\Bigr|^p\, dx\Bigr)^{p_1/p}\\
  &\le \|f\|_p^{p_1} +
  \Bigl(\int_{m/2\leq |x|\leq 2m}\Bigl|\int_{|y|\ge m} |f(y)| |K_\eta(|x-y|)|\, dy\Bigr|^p\,
  dx\Bigr)^{p_1/p} \\
  &+\Bigl(\int_{|x|\leq m/2}\Bigl|\int_{|y|\ge m} |f(y)| |K_\eta(|x-y|)|\, dy\Bigr|^p\,
  dx\Bigr)^{p_1/p}\\
  &=:\|f\|_p^{p_1}+J_{n,m,1}+J_{n,m,2}.
  \end{align*}
    Since $K_{\eta}(|\cdot|)$ is a Schwartz function, it is easily seen that for any $\ell>1$,
    \begin{equation*}
       J_{n,m,1}\leq C  \Bigl(\int_{|x|\ge m/2} |f_\ell^\ast(x)|^p\, dx\Bigr)^{p_1/p}.
    \end{equation*}
    and
    \begin{align*}
      J_{n,m,2}&\leq C m^{-\ell}  \Bigl|\int_{|y|\ge m} |f(y)| (1+|y|)^{-\ell-d-1}\, dy\Bigr|^{p_1}
  \leq C m^{-\ell}\|f\|_p^{p_1},
    \end{align*}
   where the last step uses  H\"older's inequality if $p\ge 1$, and  Nikolskii's inequality if $p<1$.
Putting the above together,  we obtain \eqref{3-14}.

\smallbreak
\underline{Step 4.} Conclude the proof of  \eqref{3-6}.

Setting $P_{n,m}^* (x)=P_{n,m}(\psi(x/n))\chi_{B(n\pi)}(x)$,  we have
\begin{align*}
 \|f\|_{L^q(\RR^d)}&\le \liminf_{m\to\infty}
 \liminf_{n\to\infty}\|P_{n,m}^*\|_{L^q(\RR^d)}= \liminf_{m\to\infty}
 \liminf_{n\to\infty}\,n^{d/q}\|P_{n,m}\|_{L^q(\SS^d)}\\
 &\le \liminf_{m\to\infty}\liminf_{n\to\infty}
 \,n^{d/q}C(n,d,p,q)
 \|P_{n,m}\|_{L^p(\SS^d)}\\
 &\le \Bigl[\liminf_{n\to\infty} \frac{C(n,d,p,q)}{n^{d(1/p-1/q)}}\Bigr]
 \Bigl[\liminf_{m\to\infty}\limsup_{n\to\infty}
 \,n^{d/p}
 \|P_{n,m}\|_{L^p(\SS^d)}\Bigr],
\end{align*}
where we used \eqref{3-11} and Fatou's lemma in the first step, \eqref{2-2-t} in the second step, and \eqref{3-9} in the third step.
However, combining \eqref{step2} with \eqref{3-14}, we get
\begin{align*}\limsup_{n\to\infty}\,(n^{d/p}\|P_{n,m}\|_p)^{p_1} \leq
    (1+C m^{-\ell})\|f\|_p^{p_1}
  +C \Bigl(\int_{|x|\ge m/2} |f_\ell^\ast(x)|^p\, dx\Bigr)^{p_1/p}
\end{align*}

which, according to Lemma \ref{lem-2-3},  goes to $\|f\|^{p_1}_p$ as $m\to\infty$. This proves \eqref{3-6}.

\section{Proof of Theorem \ref{thm-3-1}: upper estimate}

In this section,
we  will prove  that for $0<
p<\infty$,
\begin{equation}\label{4-1}
  \limsup_{n\to\infty}\frac{C(n,d, p, \infty)}{n^{d/p}}\le \mathcal{L}(d, p,\infty).
\end{equation}

Let $P_n\in\Pi_n^d$ be such that $\frac{\|P_n\|_\infty}{\|P_n\|_p}=C(n,d, p, \infty).$
Without loss of generality, we may assume that $P_n(e)=\|P_n\|_\infty=1$. For
the proof of \eqref{4-1}, it is then sufficient to prove that
\begin{equation}\label{4-2}
\liminf_{n\to\infty}\,n^{d/p}\|P_n\|_p\ge \mathcal{L}(d, p,\infty)^{-1}.
\end{equation}

The proof  \eqref{4-2} relies on  several lemmas.
The first lemma is on optimal asymptotic bounds for well-separated spherical designs, proved recently by  Bondarenko, Radchenko and
Viazovska \cite{BRV1, BRV2}.

\begin{lem}\cite{BRV1, BRV2}\label{lem-3-6}
Let $A=A_d$ be a large parameter depending only on $d$. Then for each integer
$N\ge A_d n^d$, there exists a set $\{z_{n,j}\}_{j=1}^{N}$ of $N$ points on
$\SS^d$ such that
\begin{equation}\label{2-22}
\f1{\o_d}\int_{\SS^d} P(x)\, d\sigma(x)=\frac 1{N}\sum_{j=1}^{N} P(z_{n,j}),\qquad
\forall\,P\in\Pi_{4n}^{d}
\end{equation}
and
$
\min_{1\le i\neq j\le N} \rho(z_{n,i}, z_{n,j})\ge c_d N^{-1/d}.$
\end{lem}

 The second lemma  is on the distribution of nodes of spherical designs. Denote by $B(x,\t)$ the spherical cap $\{y\in\SS^d:\  \ \rho(x,y)\leq \t\}$ with center $x\in\SS^d$ and radius $\t\in (0, \pi]$.
\begin{lem}\label{lem-Yudin}\cite{Yu-95, Yu-05, FaLe}If the formula  \eqref{2-22} holds for all $P\in \Pi_n^d$, then
$$\SS^d =\bigcup_{j=1}^{N} B(z_{n,j}, \t_n),$$
where $\t_n=\arccos t_n\sim \f1n$ and $t_n$  is the largest  root of the following  algebraic polynomial on $[-1,1]$:
$$
  Q_n(t):=\begin{cases} P_k^{(\f{d-2}2, \f{d-2}2)}(t), \    \ &\text{if $n=2k-1$,}\\
     P_k^{(\f{d-2}2, \f{d}2)}(t),   \    \ &\text{if $n=2k$}.
\end{cases}
$$
\end{lem}

The third  lemma reveals a  connection between   positive cubature formulas  and the Marcinkiewitcz-Zygmund inequality on the sphere.

\begin{lem}\label{lem-3-7}\cite[Theorem~4.2]{Dai}
Suppose that $\Lambda$ is a finite subset of $\SS^{d}$, $\lambda_\omega\ge 0$ for all
$\omega\in\Lambda$, and the  formula,
$
\int_{\SS^d}f(x)\,d\sigma(x)=\sum_{\omega\in\Lambda} \lambda_\omega
f(\omega),
$
 holds  for all $f\in \Pi_{3n}^{d}$.
 Then for $0<p< \infty$ and all $f\in\Pi_n^d$,
\[
\|f\|_p\asymp \Bigl(\,\sum_{\omega\in \Lambda} \lambda_\omega |f(\omega)|^p
\Bigr)^{1/p}
\]
with the constant of equivalence depending only on $d$ and $p$ when $p\to 0$.
\end{lem}


Now we turn to the proof of \eqref{4-2}.
Let   $\varepsilon\in (0,1)$
be  an arbitrarily given positive  parameter, and let
$\eta_1=\eta_{1,\va}\in C_c^\infty[0,\infty)$ be such that $\eta_1(x)=1$ for $x\in
[0,1]$ and $\eta_1(x)=0$ for $x\ge 1+\varepsilon$. Let $G_{n, \eta_1}$  denote the polynomial on $[-1,1]$ as  defined  in Lemma~\ref{lem-3-5}.
 Invoking Lemma \ref{lem-3-6} with $N=N_n= An^d$, we have that for
$x\in \SS^d$,
\begin{equation}\label{4-3-0}
  P_{n}(x)=\int_{\SS^d} P_n(y) G_{n,\eta_1}(x\Cdot y)\, d\sigma(y)
  =\frac {\o_d}{N_n}\sum_{j=1}^{N_n}P_n(z_{n,j}) G_{n,\eta_1}(x\Cdot z_{n,j}).
\end{equation}
According to Lemma \ref{lem-Yudin} and Lemma \ref{lem-3-6}, the set of nodes
 $\{z_{n,i}\}_{i=1}^{N_n}\subset \SS^d$  satisfies
 \begin{equation}\label{2-24}
  \min_{1\le i\neq j\le N_n}\rho(z_{n,i}, z_{n,j})\ge
\frac{\delta_d} n\  \ \text{and}\   \  \max_{x\in\SS^d}\min_{1\leq j\leq N_n} \rho(x, z_{n,j})\leq \f {c_d}n.
 \end{equation}
   Without loss of generality, we may assume that
$z_{n,1}=e$. By Lemma \ref{lem-3-7},
\begin{equation}\label{4-3}
  \Bigl(\sum_{j=1}^{N_n}|P_n(z_{n,j})|^p\Bigr)^{1/p}
  \le C n^{d/p}\|P_n\|_p=\frac{C n^{d/p}}{C(n,d, p, \infty)}\le C_d<\infty.
\end{equation}
 Next, write $z_{n,j}=\psi(y_{n,j}/n)$ for
$1\le j\le N_n$ with
  $y_{n,j}\in B(n\pi)$. Since $\rho(\psi(u), \psi(v))\leq \f \pi {\sqrt{2}} |u-v|$ for any $u, v\in B(\pi)$, we obtain from \eqref{2-24} that
\begin{equation}\label{3-24}
  \min_{1\le i\neq j\le N_n}|y_{n,i}-y_{n,j}|\ge \delta_d'>0.
\end{equation}
 Rearrange the order of the codes  $z_{n,j}$ of the spherical  design  so that $0=|y_{n,1}|\leq |y_{n,2}|\leq \cdots\leq |y_{n, N_n}|$.

  Set   $\Lambda_n:=\{ y_{n,j}\colon 1\le j\le
N_n\}$. We claim that there exists a constant $\ga_d>0$ depending only on $d$ such that  for  $m=1, \dots, n$,
\begin{equation}\label{2-27}
B(\ga_d^{-1} m)\cap \Lambda_n\subset \{ y_{n,1},\dots, y_{n,m^d}\}\subset B(\ga_d m)\cap
\Lambda_n.
\end{equation}
Indeed, noticing  that    for any $0< t\leq n$,
  \begin{equation*}
    \Bl\{j:\   \ 1\leq j\leq N_n,  \   \  \rho(z_{n,j}, e) \leq \f {t\pi}n\Br\}=\Bl\{j:\   \ 1\leq j\leq N_n, \   \  |y_{n,j}| \leq t\pi\Br\},
  \end{equation*}
we deduce from \eqref{2-24}   that  for any $1\leq t\leq n\pi$,
$$ \#\Bl\{ j:\  \ 1\leq j\leq N_n,\  \  |y_{n, j}|\leq t\Br\}\asymp t^d,$$
which together with the monotonicity of $\{|y_{n,j}|\}_{j=0}^{N_n}$ implies the claim \eqref{2-27}.

Now  the rest of the proof follows along the same line as that of \cite{LL1}.
For simplicity, we set $P_n^* (x):=P_n(\psi(x/n))$ for
$x\in\RR^d$.
Let $\mathcal{A}$ be a sequence of positive integers such that
\begin{equation*}\label{4-4}
\lim_{n\to\infty,\ n\in\mathcal{A}} n^{d/p}\|P_n\|_p=\liminf_{n\to\infty}
n^{d/p}\|P_n\|_p.
\end{equation*}
By \eqref{4-3} and \eqref{2-27}, for each fixed $m\ge 1$, we may find a subsequence
$\mathcal{T}_m$ of $ \mathcal{A}$ such that
\begin{equation}\label{4-6}
  \lim_{n\to\infty,\ n\in \mathcal{T}_m} P_n(z_{n,j})=\lim_{n\to\infty,\ n\in
  \mathcal{T}_m} P_n^*(y_{n,j})=\alpha_j\in\RR,\quad j=1,\dots, m^d,
\end{equation}
and
\begin{equation*}\label{2-30}
  \lim_{n\to\infty,\ n\in \mathcal{T}_m} y_{n,j}=y_j\in B(\ga_d m),\quad j=1,\dots,
  m^d.
\end{equation*}
We may also assume that
\[
\mathcal{A}\supset \mathcal{T}_1\supset \mathcal{T}_2\supset \ldots\supset
\mathcal{T}_m\supset \mathcal{T}_{m+1}\supset\dots,
\]
so that both the sequences $\{\alpha_j\}_{j=1}^\infty$ and
$\{y_j\}_{j=1}^\infty$ are independent of $m$. Note that $\alpha_1=P_n(e)=1$ and
$y_1=0$.

By \eqref{4-3} and \eqref{4-6}, for each $m\ge 1$,
\[
\sum_{j=1}^{m^d} |\alpha_j|^p=\lim_{n\to\infty, n\in\mathcal{T}_m}\sum_{j=1}^{m^d}
|P_n(z_{n,j})|^p\le C_d^p.
\]
Hence,
\begin{equation}\label{4-7}
  \sum_{j=1}^\infty|\alpha_j|^p\le C_d^p<\infty.
\end{equation}

Now we define
\begin{equation}\label{4-8}
f(x):=\f {\o_d}A\sum_{j=1}^\infty \alpha_j K_{\eta_1}(|x-y_j|),\quad x\in\RR^d,
\end{equation}
where $K_{\eta_1}(|\Cdot|)$ is the inverse Fourier transform of $\eta_1(|\Cdot|)$ on
$\RR^d$. According to \eqref{3-24},  $\inf_{j\neq j'}|y_j-y_{j'}|\ge
\delta'>0$. Since $K_{\eta_1}$ is a Schwartz function on $\RR^d$,
$
\sup_{x\in\RR^d}\sum_{j=1}^\infty |K_{\eta_1}(|x-y_j|)|\le C_{\eta_1}<\infty.
$
By \eqref{4-7}, this  implies that the series \eqref{4-8} converges to $f$ both
uniformly on $\RR^d$ and  in the norm of $L^p(\RR^d)$. Moreover, the
 function $f$ satisfies that
$
\|f\|_p\le C A^{-1}\Bigl(\sum_{j=1}^\infty
|\alpha_j|^p\Bigr)^{1/p}<\infty,
$
and
\begin{equation*}\label{4-9}
  \wh{f}(\xi)=\f {\o_d} A\eta_1(|\xi|)\Bigl(\sum_{j=1}^\infty \alpha_j e^{- i y_j\cdot \xi}\Bigr),
\end{equation*}
where    the infinite series converges  in a distributional sense.   According to the Paley-Wiener theorem,   $f$ extends to
 an entire function on $\CC^d$  of spherical exponential type $1+\va$. In particular, this implies that \eqref{4-9} implies that
 the function
\[
f_\varepsilon(x):=(1+\varepsilon)^{-d} f\Bigl(\frac x{1+\varepsilon}\Bigr)
\]
belongs to the space
$ \mathcal{E}_{p}^{d}$.

To complete the proof of \eqref{4-2}, we
need  the following technical lemma:

\begin{lem}\label{lem-4-1} Let $\ga_d$ denote the constant in \eqref{2-27}. If
 $r\ge 1$ and  $m\ge 2\ga_d r$, then for  any $\ell\ge 1$,
\begin{equation}\label{4-11-0}
  \limsup_{n\to\infty, n\in \mathcal{T}_m} \sup_{x\in B(r)}\Bigl|P_n^* (x)-
  A^{-1}\omega_{d}\sum_{1\le j\le m^d} \alpha_j K_{\eta_1}(|x-y_j|)\Bigr|\le
  C_{\eta_1, \ell, r} m^{-\ell}.
\end{equation}
In particular, this implies that
\begin{equation}\label{4-12-0}
  f(0)=A^{-1}\omega_{d}\sum_{j=0}^\infty \alpha_j K_{\eta_1}(|y_j|)=1.
\end{equation}

\end{lem}

For the moment, we take Lemma \ref{lem-4-1} for granted and proceed with the proof of \eqref{4-2}.

Since $f_\va \in\mathcal{E}_p^d$,  we have
\[
  (1+\varepsilon)^{-d}=|f_\varepsilon (0)|\le \|f_\varepsilon\|_\infty \le
  \mathcal{L}(d, p,\infty)\|f_\varepsilon\|_p=\mathcal{L}(d, p,\infty) (1+\varepsilon)^{-d/p'}\|f\|_p.
\]
Thus,
\begin{equation}\label{2-35}
  \|f\|_p\ge \mathcal{L}(d,p,\infty)^{-1} (1+\varepsilon)^{-d/p}.
\end{equation}

 On the other hand,  setting $p_1=\min\{p,1\}$, we obtain that for $n\in\mathcal{T}_m$,
\begin{align*}
(n^{d/p}\|P_n\|_p)^{p_1}&\ge \Bigl(n^d \int_{B(e, \frac rn)} |P_n(x)|^p\,
d\sigma(x)\Bigr)^{p_1/p}=\Bigl(\int_{B(r)} |P_n^*
(x)|^p\Bigl(\frac{\sin (|x|/n)}{|x|/n}\Bigr)^{d-1}\, dx\Bigr)^{p_1/p}
\\
&\ge \Bigl(\int_{B(r)}\Bigl| \f {\o_d}A\sum_{j=1}^{m^d} \alpha_j
K_{\eta_1}(|x-y_j|)\Bigr|^p\Bigl(\frac{\sin (|x|/n)}{|x|/n}\Bigr)^{d-1}\,
dx\Bigr)^{p_1/p}
\\
&-C_r \sup_{x\in B(r)}\Bigl|P_n^* (x)- (A^{-1}\omega_{d})\sum_{j=1}^{m^d} \alpha_j
K_{\eta_1}(|x-y_j|)\Bigr|^{p_1}.
\end{align*}
 It then follows from Lemma \ref{lem-4-1}   that for any $\ell>1$,
\begin{align*}
\liminf_{n\to\infty}\,(n^{d/p}\|P_n\|_p)^{p_1}
&=\lim_{n\to\infty,\
n\in\mathcal{T}_m}\,(n^{d/p}\|P_n\|_p)^{p_1} \\
&\ge \Bigl(\int_{B(r)}\Bigl| \f{\o_d}A\sum_{j=1}^{m^d} \alpha_j
K_{\eta_1}(|x-y_j|)\Bigr|^p\, dx\Bigr)^{p_1/p} - C_r m^{-\ell}.
\end{align*}
Letting $m\to\infty$, we obtain from \eqref{4-8} and  the dominated convergence theorem that
\[
\liminf_{n\to\infty}\,n^{d/p}\|P_n\|_p\ge
\Bigl(\int_{B(r)}|f(x)|^p\, dx\Bigr)^{1/p}.
\]
Letting $r\to\infty$, and using \eqref{2-35}, we then deduce
\[
\liminf_{n\to\infty}\,n^{d/p}\|P_n\|_p\ge \|f\|_p \ge
\mathcal{L}(d, p,\infty)^{-1}(1+\varepsilon)^{-d/p}.
\]
 Now  the desired estimate \eqref{4-2} follows by letting $\varepsilon\to 0$. \hb

It remains to prove Lemma \ref{lem-4-1}.

\begin{proof}[Proof of Lemma \ref{lem-4-1}]
Note first that  by Lemma \ref{lem-3-5},
\[
\lim_{n\to\infty}\frac 1{n^{d}}\,G_{n,\eta_1} \Bigl(\psi\Bigl(\frac
x{n}\Bigr)\Cdot \psi\Bigl(\frac{y}{n}\Bigr)\Bigr)=K_{\eta_1}(|x-y|)
\]
holds uniformly for $x, y\in B(\ga_d m)$. Note also that for $1\le j\le m^d$,
\begin{align*}
&\Bigl|n^{-d} G_{n,{\eta_1}} \Bigl(\psi\Bigl(\frac x{n}\Bigr)\Cdot
\psi\Bigl(\frac{y_{n,j}}{n}\Bigr)\Bigr)- K_{\eta_1}(|x-y_j|)\Bigr|\\
&\qquad \le \sup_{z\in B(\ga_d m)}\Bigl|n^{-d} G_{n,{\eta_1}}\Bigl(\psi\Bigl(\frac x{n}\Bigr)\Cdot
\psi\Bigl(\frac{z}{n}\Bigr)\Bigr)-K_{\eta_1}(|x-z|)\Bigr|\\
&\qquad +\Bigl|K_{\eta_1}(|x -y_{n,j}|)-K_{\eta_1}(|x-y_j|)\Bigr|.
\end{align*}
Letting $n\to\infty$ and $n\in\mathcal{T}_m$, we conclude that for $1\le j\le m^d$,
\begin{align}\label{4-11}
\lim_{n\to\infty, n\in\mathcal{T}_m} \sup_{x\in B(\ga_d m)}\,\Bigl|\frac 1{n^d}\,
G_{n,{\eta_1}} \Bigl(\psi\Bigl(\frac x{n}\Bigr)\Cdot  \psi\Bigl(\frac{y_{n,j}}{n}\Bigr)\Bigr)-
K_{\eta_1}(|x-y_j|)\Bigr|=0.
\end{align}

Next, using \eqref{4-3-0}, we obtain that for $x\in B(r)$ and $m\ge 2r\ga_d$,
\begin{align*}
  P^*_{n}(x)
  &=\frac 1{A n^d}\sum_{j=1}^{A n^d}P_n\bigl(z_{n,j}\bigr)
  G_{n,{\eta_1}}\Bigl(\psi\Bigl(\frac xn\Bigr)\Cdot \psi(\frac{y_{n,j}}n)\Bigr)\\
  &=\sum_{1\le j\le m^d} + \sum_{m^d<j\le A n^d} =:I_{n,m}(x)+J_{n,m}(x).
\end{align*}

To estimate the second term $J_{n,m}(x)$, we note that
$\rho(\psi(\frac{y_{n,j}}{n}), e)\ge \frac{\ga_d^{-1} m}n\ge \frac{2r} n$ for $m^d\le
j\le A n^d$, and $\rho(\psi(\frac xn), e)\le \frac rn$ for $x\in B(r)$. Thus,
$\rho\bigl(\psi(\frac xn), \psi(\frac{y_{n,j}}n)\bigr)\ge \frac{c_d m}{n}$ for
$x\in B(r)$ and $m^d\le j\le A n^d$. It follows by \eqref{4-3} that for any $\ell\ge 1$ and
$x\in B(r)$,
\begin{equation}\label{4-12}
  |J_{n,m}(x)|\le C m^{-\ell}\Bl(\sum_{j=1}^{A n^d} |P_n(z_{n,j})|^p\Br)^{1/p}\le C
  m^{-\ell}n^{d/p}\|P_n\|_p \le C m^{-\ell}.
\end{equation}
For the term $I_{n,m}$, we use \eqref{4-11} and \eqref{4-6} to obtain
\begin{equation}\label{4-13}
  \lim_{n\to\infty,\ n\in\mathcal{T}_m} I_{n,m}(x)=A^{-1}\omega_{d}\sum_{1\le
  j\le m^d} \alpha_j K_{{\eta_1}}(|x-y_j|)\quad \text{uniformly for \ $x\in B(r)$}.
\end{equation}
Combining \eqref{4-12} with \eqref{4-13}, we conclude that
\[
  \limsup_{n\to\infty, n\in \mathcal{T}_m} \sup_{x\in B(r)}\,\Bigl|P_n^* (x)-
  A^{-1}\omega_{d}\sum_{1\le j\le m^d} \alpha_j K_{\eta_1}(|x-y_j|)\Bigr|\le C
  m^{-\ell}.
\]
This proves \eqref{4-11-0}.

Finally, invoking \eqref{4-11-0} with $x=0$, and recalling that $P_n^*
(0)=P_n(e)=1$, we obtain
\[
  \Bigl|1- A^{-1}\omega_{d}\sum_{0\le j\le m^d} \alpha_j K_{\eta_1}(|y_j|)\Bigr|\le
  C m^{-2}\quad \forall\,m\ge 1.
\]
Letting $m\to\infty$, we obtain  \eqref{4-12-0}. This completes the proof.
\end{proof}

\section{Proof of Theorem \ref{thm1-2}}

  We  break the proof of Theorem \ref{thm1-2} into two parts.  In the first part, we prove the following proposition, which gives  the exact value of the Nikolskii constant for nonnegative functions from the class $\mathcal{E}_1^d$ on $\RR^d$.

 \begin{prop}\label{prop-5-1}We have
 \begin{equation}\label{5-1}\sup_{0\leq f\in\mathcal{E}_1^d}\f{\|f\|_{L^\infty(\RR^{d})}}{\|f\|_{L^1(\RR^d)}}
  =
\frac1{4^d \pi^{d/2}\Ga(d/2+1)}.\end{equation}
 \end{prop}

 In the second part, we compute the exact value of the Nikolskii constant for nonnegative polynomials on $\SS^d$:
 \begin{prop}\label{prop-5-2}For $n=1,2,\cdots$,
 \begin{equation}\label{C-P-}
\sup_{0\leq P\in\Pi_n^d} \frac{\|P\|_\infty}{ \|P\|_1}=\o_d^{-1}
\begin{cases}
\f{(2k+d)(k+d-1)!}{k! d!},&n=2k,\\ 2\binom{d+k}{d},&n=2k+1.
\end{cases}
\end{equation}

 \end{prop}

Note that \eqref{C-P-}, in particular,  implies
 \begin{equation}\label{5-3}
\lim_{n\to\infty}\sup_{0\leq P\in\Pi_n^d} \frac{\|P\|_\infty}{n^d \|P\|_1}=\frac1{4^d \pi^{d/2}\Ga(d/2+1)}.
\end{equation}
Theorem \ref{thm1-2} is a direct consequence of  \eqref{5-3} and  \eqref{5-1}.

We point out that \eqref{C-P-} for algebraic polynomials on intervals was obtained in  \cite{Lev98}.
 Proofs of Propositions \ref{prop-5-1} and  \ref{prop-5-2} are given in the next two subsections respectively.

 \subsection{Proof of Proposition \ref{prop-5-1}}

 For simplicity, we set
$$\mathcal{L}^+:=\sup
\{\|f\|_\infty:\ \ 0\leq  f\in\mathcal{E}_1^d, \  \    \|f\|_1=1\}.$$
To show    the lower  estimate, $$\mathcal{L}^+\ge \frac1{4^d \Ga(d/2+1)\pi^{d/2}},$$
 we consider the  function
$f(x):=\bl(j_{d/2}({|x|}/2)\br)^2$.
 Note that  $G(x):=\frac{\omega_{d-1}}{d(2\pi)^{d}}\,
j_{d/2}(|x|)$  is   the inverse Fourier transform of $\chi_{\BB^d}$, where $\BB^d$ denotes the unit ball centered at the origin in $\RR^d$.
 It follows that  $0\leq f\in\mathcal{E}_1^d$. Furthermore,  by
Plancherel's theorem,
\begin{align*}
 \|f\|_1& =\Bl(\f {d (2\pi)^d}{\o_{d-1}}\Br)^2 2^d \int_{\RR^d} |G(x)|^2 \, dx=\f {d^2 2^d (2\pi)^d}{(\o_{d-1})^2}\int_{\BB^d} 1\, dx\\
  &=\f {2^d d (2\pi)^{d}}{\o_{d-1}}= 2^{d-1}\o_d d!.
\end{align*}
This yields the stated lower estimate:
 \begin{align*}
    \mathcal{L}^{+} \ge \f {f(0)}{\|f\|_1} =\f {1}{ 2^{d-1} d! \o_d}=\frac1{4^d \Ga(d/2+1)\pi^{d/2}}.
 \end{align*}

To show  the upper estimate,
\begin{equation}\label{5-4}
 \mathcal{L}^+\leq  \frac1{4^d \Ga(d/2+1)\pi^{d/2}},
  \end{equation}
  we need  the following Markov type quadrature formula, which was established in \cite{GI15}:
 \begin{lem}\label{lem-5-1}\cite{GI15} Assume that  $\al\ge -\f12$ and $\tau>0$. Let
 $\mathcal{B}_{\al,\tau}$ denote the set of all even entire functions  $f$  of exponential type  $\leq 2\tau$  such that
$\int_{0}^{\infty}|f(t)|t^{2\alpha+1}\,dt<\infty$.
 Then there exists a sequence $\{\rho_k\}_{k=0}^\infty$ of positive numbers with $\rho_{0}=2^{2\alpha}(\Gamma(\alpha+1))^2(2\alpha+2)/\tau^{2\alpha+2}$ such that
\begin{equation*}
\int_{0}^{\infty}f(t)t^{2\alpha+1}\,dt=
\rho_{0}f(0)+\sum_{k=1}^{\infty}\rho_{k}f(q_{\alpha+1,k}/\tau),\   \   \forall f\in\mathcal{B}_{\al,\tau},
\end{equation*}
where the infinite  series converges absolutely, and  $\{q_{\al+1, k}\}_{k=1}^\infty$ is the sequence of all positive zeros of the Bessel function  $J_{\al+1}(x)$ arranged in increasing order.
\end{lem}

Now we turn to the proof of the estimate \eqref{5-4}. Given $\va\in (0,1)$,
let $f\in\mathcal{E}_{1}^{d}$ be a nonnegative function  such that $\|f\|_{1}=1$ and $\|f\|_\infty\ge \mathcal{L}^+-\va$.
Without loss of generality, we may assume that
$\|f\|_{\infty}=f(0)$. Define a nonnegative  radial function $g$ by
\[
g(x)=g_0(|x|):=\f 1{\o_{d-1}}\int_{\SS^{d-1}}f(|x|\xi)
\,d\sigma(\xi),\qquad x\in\RR^d.
\]
It is easily seen that
$$\wh{g}(x):=\f 1{\o_{d-1}}\int_{\SS^{d-1}}\wh{f}(|x|\xi)
\,d\sigma(\xi),\qquad x\in\RR^d,$$
 $g(0)=f(0)$, and
$
\|g\|_{1}= \|f\|_{1}=1.
$
By the Paley-Wiener theorem, this in particular implies  $g\in\mathcal{E}_1^d$.
Thus, we may
apply  Lemma \ref{lem-5-1}  to
the function $g_0$ with  $\tau=1/2$ and $\al=d/2-1$.   Taking into account the facts that $\rho_j\ge 0$ for $j=0,1,\cdots$ and $g_0$ is nonnegative, we then  obtain
\[
1=\|g\|_{1}=\omega_{d-1}\int_{0}^{\infty}g_0(t)t^{d-1}\,dt\ge
\o_{d-1}\rho_{0}g(0)=\omega_{d-1}2^{2d-2}d
(\Gamma(d/2))^2  f(0).
\]
Thus,
\begin{align*}
   \mathcal{L}^+-\va\leq  f(0) &\leq \f 1{\o_{d-1} 2^{2d-2} d (\Ga(d/2))^2}=\f 1{d! 2^{d-1} \o_d}=\frac1{4^d \Ga(d/2+1)\pi^{d/2}}.
\end{align*}
Letting $\va\to 0$ yields the desired estimate \eqref{5-4}.
\subsection{Proof of Proposition \ref{prop-5-2}}
Without loss of generality, we may assume $n=2k$. (The case $n=2k+1$ can be treated similarly). The proof follows along the same line as that of Proposition \ref{prop-5-1}.

To  show the lower estimate, \begin{equation}\label{lower}
\sup_{0\leq P\in\Pi_{2k}^d} \frac{\|P\|_\infty}{ \|P\|_1}\ge \o_d^{-1}
\f{(2k+d)(k+d-1)!}{k! d!},
\end{equation}
 we consider the polynomial
$$ f(x):=\Bl[ R_k^{(\f d2, \f{d-2}2)}(x\cdot e)\Br]^2,\  \ x\in\SS^d,$$
where $e\in\SS^d$ is a fixed point on $\SS^d$ and $R_k^{(\al,\b)}(t)=P_k^{(\al,\b)}(t)/P_k^{(\al,\b)}(1)$. Clearly,  $f\in \Pi_n^d$, and  $\|f\|_\infty =f(e)=1$. Moreover,  using \eqref{reproduce}, we have
$$ \|f\|_1=\int_{\SS^d} |R_k^{(\f d2, \f{d-2}2)}(x\cdot e)|^2\, d\s(x) =\f 1{d_k^2} \sum_{j=0}^k \f {\text{dim} \, \mathcal{H}_j^d}{\og_d}=\f {\og_d} {\text{dim}\, \Pi_k^d}.$$
It then follows from \eqref{dimension} that
\begin{align*}
  \f{\|f\|_\infty}{\|f\|_1} =\f {\text{dim}\, \Pi_k^d} {\og_d}=\f 1{\og_d} \f {(2k+d)\Ga(k+d)}{k! d!},
\end{align*}
which shows the lower estimate \eqref{lower}.

The proof of  the upper estimate, \begin{equation}\label{upper-0}
\sup_{0\leq P\in\Pi_{2k}^d} \frac{\|P\|_\infty}{ \|P\|_1}\leq  \o_d^{-1}
\f{(2k+d)(k+d-1)!}{k! d!},
\end{equation}
 relies on the following Jacobi-Gauss-Radau quadrature rules, which can be found in \cite[p. 81]{STW}:
\begin{lem}\label{Jacobi-Gauss-Radau} \cite{STW}  Let $\{x_j\}_{j=1}^N$ be the zeros of the Jacobi polynomial $P_N^{(\a+1, \be)}$ with   $\al,\b>-1$.   Then for every algebraic polynomial $P$ of degree at most $2N$,
\begin{equation*}
   \int_{-1}^1 P(x) (1-x)^\al(1+x)^\b\, dx =\ld_0 P(1) +\sum_{j=1}^N \ld_j P(x_j),
\end{equation*}
where
\begin{align*}
  \ld_0&=\f {2^{\al+\b+1} (\al+1) (\Ga (\al+1))^2 N! \Ga(N+\b+1)}{\Ga(N+\al+2)\Ga(N+\al+\be+2)},\\
  \ld_j&=\f { 2^{\al+\b+4} \Ga(N+\al+2) \Ga(N+\b+1)}{ (1+x_j)(1-x_j)^2[P_{N-1}^{(\al+2, \b+1)}(x_j)]^2 (N+\al+\b+2) \Ga(N+\al+\b+3)}.
\end{align*}
\end{lem}

To show \eqref{upper-0}, let $f$ be an arbitrary nonnegative spherical polynomial of degree at most $2k$   such that $\|f\|_\infty =f(x_0)=1$ for some $x_0\in\SS^d$. Without loss of generality, we may assume that $x_0=(1,0,\cdots, 0)$. Define
$$ g(t):=\f 1{\og_{d-1}}\int_{\sph} f(t, \sqrt{1-t^2} \xi)\, d\s(x),\  \ t\in[-1,1].$$
It is easily seen that $g$ is an algebraic polynomial of degree at most $2k$ on $[-1,1]$, $g(1)=f(x_0)=1$, and
$$\int_{-1}^1 g(t) (1-t^2)^{\f {d-2}2}\, dt
 =\f 1{\og_{d-1}}\int_{-1}^1 \int_{\sph} f(t, \sqrt{1-t^2}\xi)\, d\s(\xi) (1-t^2)^{\f{d-2}2}\, dt=\f 1{\og_{d-1}}\|f\|_{L^1(\SS^d)}.$$
 Using Lemma \ref{Jacobi-Gauss-Radau} with $\al=\b=\f{d-2}2$, and taking into account the facts that $g$ is nonnegative and $\ld_j>0$ for $j=0,1,\cdots$, we deduce
 \begin{align*}
   \|f\|_{L^1(\SS^d)}& =\og_{d-1}\int_{-1}^1 g(t) (1-t^2)^{\f {d-2}2}\, dt \ge \ld_0 \og_{d-1} g(1) \\
   &= \og_{d-1} \f { 2^{d-2} d (\Ga(d/2))^2 k!}{(k+d/2) \Ga(k+d)}=\og_d \f { k! d!} { (2k+d) \Ga(k+d)}.
 \end{align*}
 Thus,
 $$\f{\|f\|_\infty}{\|f\|_{L^1(\SS^d)}} \leq \og_d\f {(2k+d) \Ga(k+d)}{k! d!},$$
 and the upper estimate \eqref{upper-0} then follows.

\end{document}